\documentclass[a4paper,10pt]{amsart}
\usepackage{algorithmic}
\usepackage{algorithm}
\usepackage{mathtools}
\usepackage{pgfplots}

\newtheorem{theorem}{Theorem}
\newtheorem{definition}{Definition}

\newtheorem*{assumption*}{Assmuption}

\title{Irreducibility criterion for singular hypersurfaces of $(\mathbb{C}^n,0)$}
\author{P. Fortuny Ayuso}
\address{Dpto. Matemáticas, Universidad de Oviedo. Oviedo, Spain.}
\email{fortunypedro@uniovi.es}
\date{\today}

\usepackage{enumerate}

\definecolor{nuevo}{RGB}{0,0,120}

\begin{document}
\newcommand{\mex}[1]{\ensuremath{\left\lceil #1 \right\rceil}}
\strut
  \begin{large}
\begin{center}
    \textsc{Statement from the author}
  \end{center}
  
  This paper is wrong. The main idea is probably interesting but the statement (and obviously, the proof) are wrong. I thank Y. Genzmer for pointing out the issues which led me to realizing the mistake.

  I am leaving the contents as I wrote them in case someone finds them relevant, even if only for avoiding the same mistakes.

  Essentially: the Weierstrass derivation is wrongly computed everywhere.

  The simplest example (Genzmer) is a homogeneous cubic $x^3 + xy^2 + y^{3}$.
\end{large}
\begin{center}
  Pedro Fortuny Ayuso, June 5 2021.
\end{center}
\thispagestyle{empty}
\newpage
\begin{abstract}
Using vector fields we obtain an irreducibility criterion for hypersurfaces. It only requires the Weierstrass division.
\end{abstract}
\maketitle

\section{Introduction and Notation}
The problem of deciding whether a plane curve (in implicit form, obviously) is irreducible is of paramount importance for studying its singularity (see
\cite{Abhyankar-adv-1989,abhyankar-moh,moh-1975,kuo-1977,evelia-2012} just for several relevant examples, without any aim to completeness). The usual ways to solve it in characteristic $0$ are by means of the Newton diagram (and Puiseux series) or using Approximate Roots, following Abhyankar and Moh. Both ways require ``going further'' than the first Puiseux exponent, and a very delicate analysis of the singularity.

In this short note we provide a criterion which, using elementary vector fields and Weierstrass division decides whether a germ of (reduced) \emph{hypersurface} is or not irreducible. The result is based on the fact that vector fields can have ``bad order of tangency'' with a hypersurface if and 
only if this is reducible, because the contact structure associated to vector fields is essentially different in the reducible and irreducible cases.

\section{Irreducibility criterion for hypersurface singularities}
Let $f:(\mathbb{C}^n,0)\rightarrow (\mathbb{C},0)$ define a reduced germ of hypersurface $f=0$. For simplicity, we denote $(x,\overline{y})$ a system of coordinates in $(\mathbb{C}^n,0)$ such that $f(x,\overline{y})$ is in Weierstrass form:
\begin{equation}
  \label{eq:weierstrass}
  f(x,\overline{y}) = x^k + \sum_{i=0}^{k-1} F_{i}(y)x^i.
\end{equation}
where $k$ is the multiplicity of $f$. From here on, we assume that $k>2$ (the case $k=2$ is trivial) and $F_0(\overline{y})\neq 0$ (the hypersurface does not contain $x=0$). Let $df$ denote the differential form of $f$ and $X\in {\mathfrak X}(\mathbb{C}^2,0)$ a germ of holomorphic vector field in $(\mathbb{C}^{2},0)$.
\begin{definition}
  The \emph{tangency function} of $X$ with $f(x,\overline{y})=0$ (with respect to the specific coordinate function) is the remainder $R(x,\overline{y})$ of the Weierstrass division:
  \begin{equation*}
    df(X) = Q(x,\overline{y})f(x,\overline{y}) + R(x,\overline{y}).
  \end{equation*}
  The \emph{tangency order} of $X$ with $f(x,\overline{y})$ \emph{in the $x$-direction}  is the order of $R(x,\overline{y})$ as a power series in $x$.
\end{definition}
The tangency function measures, in some sense, ``how'' $X$ fails to be tangent to $f(x,\overline{y})=0$ \cite{Fortuny-Ribon-Canadian,fortuny-racsam,fortuny-normal-forms}, the tangency order the order of that ``failure''. We shall omit the qualifier ``in the $x$-direction'' because it is unnecessary in what follows.

The irreducibility criterion is the following (recall that $x$ does not divide $f(x,\overline{y})$):
\begin{theorem}
  With the notations and hypotheses above (recall that $f$ is reduced and $k>2$), then $f(x,\overline{y})=0$ is reducible if and only if there exists an integer $2\leq r<k$ and a vector field
  \begin{equation*}
    X = x^r \frac{\partial }{\partial x}
  \end{equation*}
  whose tangency order with $f(x,\overline{y})$ is $0$.
\end{theorem}
\begin{proof}
  Assume $f(x,\overline{y})=0$ is reducible and set $f(x,\overline{y})=f_1(x,\overline{y})f_2(x,\overline{y})$. As $f$ is in Weierstrass form, we can assume $f_1(x,\overline{y})$ is too and $f_2(x,\overline{y})$ is almost: 
  \begin{equation*}
    f_1(x,\overline{y}) = x^{a_1} + \sum_{i=0}^{a_1-i}F^{1}_{i}(\overline{y})x^i,
    \;
    f_2(x,\overline{y}) = u (x,\overline{y})\bigg(x^{a_2} +
    \sum_{i=0}^{a_2-i}F^2_i(\overline{y})x^i\bigg).
  \end{equation*}
Let $X$ be the vector field
  \begin{equation*}
    X = x^{a_1+1} \frac{\partial }{\partial x}.
  \end{equation*}
  As $f_1(x,\overline{y})$ has degree $a_1$ in $x$, performing the Weierstrass division, we obtain $x^{a_1+1} = xf_1(x,\overline{y}) + R(x,\overline{y})$ where $R(x,\overline{y})=-xF^{1}_0(y)+x^2S(x,\overline{y})$ with $S(x,\overline{y})$ a holomorphic function. From this:
  \begin{equation*}
    df(X) = \left(\left(
      f_2 \frac{\partial f_1}{\partial x} +
      f_1 \frac{\partial f_2}{\partial x} 
    \right)dx + (\cdots) \right)
  \left(xf_1 + R\right) \frac{\partial }{\partial x}
\end{equation*}
where the dots denote an irrelevant holomorphic $1$-form. This gives:
\begin{equation}\label{eq:dfX-divided}
  df(X) =  xf_1f_2 \frac{\partial f_1}{\partial x} + xf_1^2 \frac{\partial f_2}{\partial x} 
  + x(-F^1_0(\overline{y}) + xS(x,\overline{y})) \left( f_1 \frac{\partial f_2}{\partial x} +
    f_2 \frac{\partial f_1}{\partial x}\right).
\end{equation}
The first term is a multiple of $f$. The second one is:
\begin{equation*}
  xf_1^2
  \left(u\frac{\partial (f_2/u)}{\partial x} + (f_2/u) \frac{\partial u}{\partial x}\right) =
  v(x,\overline{y})\left(x^{2a_1+a_2} + \sum_{i=1}^{2a_1+a_2-1} h_i(\overline{y})x^i\right)
\end{equation*}
(for a unit $v(x,\overline{y})$). This term is not a multiple of $f$ because $f$ is reduced. Notice the indices starting at $1$. The last term in \eqref{eq:dfX-divided} has degree $a_1+a_2$ in $x$ as a Weierstrass polynomial (except for a unit). Performing the Weierstrass division, we obtain:
\begin{equation*}
  df(X) = Q(x,\overline{y})f(x,\overline{y}) -
  w(x,\overline{y})F^1_0(\overline{y}))^{2}F_0^2(\overline{y}) + xT(x,\overline{y})
\end{equation*}
for a unit $w(x,\overline{y})$ and some holomorphic function $T(x,\overline{y})$. Thus, the contact order of $X$ with $f=0$ is $0$.

The reciprocal: If $f(x,\overline{y})$ is irreducible of multiplicity $n$ then, for any $X=x^r \frac{\partial }{\partial x}$ with $2\leq r<n$
\begin{equation*}
  df(X) = \left(
    \frac{\partial f}{\partial x} dx + \cdots
  \right) x^r \frac{\partial }{\partial x } =
  x^r \frac{\partial f}{\partial x}
\end{equation*}
whose contact order with $f$ is exactly $r-1$ (as $x$ does not divide $f$).
\end{proof}

%\bibliographystyle{plain} 
%\bibliography{../../biblioams.bib}
\end{document}